\documentclass[a4paper,11pt]{amsart}
\usepackage{a4wide}
\usepackage[utf8x]{inputenc}
\usepackage[T1]{fontenc}
\usepackage{gensymb}
\usepackage{amsfonts}
\usepackage{mathrsfs}
\usepackage{amsthm,amsrefs}
\usepackage{thmtools}
\usepackage{charter}
\usepackage[all]{xy}
\usepackage{xcolor}
\usepackage{enumerate}
\usepackage[pdftex]{graphicx}
\usepackage[pdftex,
colorlinks=true,
linkcolor=blue,
citecolor=green,
hyperindex,
plainpages=false,
bookmarks=true,
bookmarksopen=true,
bookmarksnumbered,]{hyperref}

\theoremstyle{plain}
\newtheorem{theorem}{Theorem}[section]
\newtheorem{corollary}[theorem]{Corollary}
\newtheorem{lemma}[theorem]{Lemma}

\theoremstyle{definition}
\newtheorem{definition}[theorem]{Definition}
\theoremstyle{remark}

\newtheorem*{theorem*}{Theorem \ref{mainthm}}

\numberwithin{equation}{section}

\newcommand{\id}{{\rm id}}
\newcommand{\wc}{\rightharpoonup} 
\newcommand{\grass}[2]{\mathbf{G}(#1,#2)}
\newcommand{\Var}{\mathbf{V}}     
\newcommand{\var}{\mathbf{v}}     

\newcommand{\HM}{\mathcal{H}}
\newcommand{\ud}{\ensuremath{\,\mathrm{d}}}
\DeclareMathOperator{\VarTan}{VarTan}   
\DeclareMathOperator{\Tan}{Tan}
\DeclareMathOperator{\Lip}{Lip}
\DeclareMathOperator{\dist}{dist}
\DeclareMathOperator{\ap}{ap} 
\newcommand{\scale}[1]{\boldsymbol{\mu}_{#1}}

\newcommand{\mr}{ \mathop{ \rule[1pt]{.5pt}{6pt} \rule[1pt]{4pt}{0.5pt} }\nolimits }

\begin{document}
\title{A note on the convergence of almost minimal sets}
\author{Yangqin Fang}
\address{Yangqin Fang \\ Max-Planck-Institut f\"ur Gravitationsphysik,
Am M\"uhlenberg 1, 14476 Potsdam, Germany}
\email{yangqin.fang@aei.mpg.de}

\date{}
\maketitle
\begin{abstract}
	In this paper, we will show that Hausdorff convergence and varifold
	convergence coincide on the class of almost minimal sets.
\end{abstract}
\section{Introduction and notation}
We see obvious in general Huasdorff distance convergence for a sequence of sets 
do not implies the varifold convergence of the associate sequence of varifolds. 
But in this paper, we will show that the implication is true in case that
restrict on general quasiminimal sets with the total Huasdorff measure of the 
sequence tending to the total Huasdorff measure of the limit set. It is also 
true by replacing the total Huasdorff measure with the integral of an elliptic
integrand. As a consequence, we may get that the Hausdorff convergence and varifold
convergence coincide on almost minimal sets. 

An $m$-varifold on an open subset $ U\subseteq\mathbb{R}^{n}$ is a Radon
measure on $ U\times \grass{n}{m}$. We denote by $\Var_{m}( U)$ the
collection $m$-varifolds on $ U$. It can be equipped with a weak topology
given by saying that  $ V_i\wc V $ if 
\[
	\int \varphi \ud V_i\to \int \varphi \ud V,
\]
for all compactly supported, continuous real valued function $\varphi$ on
$ U\times \grass{n}{m}$.

Given any varifold $V$, we can get a corresponding Radon measure $\|V\|$ on 
$ U$ defined by 
\[
	\|V\|(A)=V(A\times \grass{n}{m}),\ \text{for }A\subseteq U.
\]
For any Borel regular $\mu$ measure on $U$ and $x\in U$, we let
$\Theta_{\ast}^{m}(\mu,x)$ and $\Theta^{\ast m}(\mu,x)$ be the lower and upper
$m$-density of $\mu$ at $x$, see \cite{Allard:1972}, if they are equal, we 
will denote it by $\Theta^{m}(\mu,x)$, called the $m$-density. For any set
$E\subseteq U$, $\Theta^{m}(E,x)$ is understood as the $m$-density of $\HM^m\mr
E$ at $x$.

A subset $E\subseteq\mathbb{R}^{n}$ is called $m$-rectifiable, if there exists a 
sequence of Lipschitz mappings $f_{i}:\mathbb{R}^{m}\to\mathbb{R}^{n}$ such that 
\[
	\HM^m\left( E\setminus f_{i}(\mathbb{R}^{m}) \right)=0.
\]
$E$ is called purely $m$-unrectifiable (or $m$-irregular) if 
$\HM^m(E\cap F)=0$ for any $m$-rectifiable set $F$, see for example  
\cite[Definition 15.3]{Mattila:1995} or \cite[3.2.14]{Federer:1969}.

Let $E\subseteq\mathbb{R}^{n}$ be a $m$-rectifiable set, $x\in E$ be any point.
An $m$-plane $\pi$ is called an approximate tangent plane if 
\[
	\limsup_{r\to 0}r^{-m}\HM^d(E\cap B(x,r))>0
\]
and for any $\varepsilon>0$,
\[
	\lim_{r\to 0}r^{-m}\HM^d(E\cap B(x,r)\setminus
	\mathcal{C}(x,\pi,r,\varepsilon))=0,
\]
where $\mathcal{C}(x,\pi,r,\varepsilon)=\{ y\in B(x,r)\mid \dist(y-x,\pi)\leq
\varepsilon |y-x| \}$. An $m$-plane $\pi$ is called a (true) tangent plane if 
for any $\varepsilon>0$, there exists $r_{\varepsilon}>0$ such that 
\[
	E\cap B(x,r)\subseteq \mathcal{C}(x,\pi,r,\varepsilon) \text{ for
	}0<r<r_{\varepsilon}.
\]
We will denote by $\Tan(E,x)$ the tangent plane of $E$ at $x$, if it exists. 

Let $E\subseteq  U$ be an $m$-rectifiable set. Then for $\HM^m$ almost every
$x\in E$, there is an unique approximate tangent plane of $E$ at $x$, see for
example \cite[Theorem 3.2.19]{Federer:1969} or \cite[Theorem 15.11]{Mattila:1995}. 
If additionally $E$ is local Ahlfors regular, that is, there exists $C\geq 1$ 
and $r_0>0$ such that for any $x\in E$, 
$0<r<r_0$ with $B(x,2r)\subseteq U$, we have that 
\begin{equation}\label{eq:AR}
	C^{-1}r^m\leq \HM^m(E\cap B(x,r))\leq Cr^m,
\end{equation}
then in this case, every approximate tangent plane is a true tangent plane.

Let $E\subseteq U$ be any set such that $\HM^m(E\cap K)<\infty$ for any
compact sets $K\subseteq U$. We define the associates varifold $\var(E)$,
by setting
\[
	\var(E)(\beta)=\int_{E_{rec}}\beta(x,\Tan(E_{rec},x))\ud\HM^d(x)+
	\int_{E_{irr}}\beta(x,T)\ud \gamma_{n,m}(T)\ud\HM^m(x)
\]
for any continues function $ \beta: U\times \grass{n}{m}\to \mathbb{R}$ with compact
support, where we decompose $E$ as the union $E_{rec}\cup E_{irr}$, $E_{rec} $ is $m$-rectifiable,
$E_{irr}$ is purely $m$-unrectifiable, $\gamma_{n,m}$ denotes the Haar measure 
on $\grass{n}{m}$.

On the power set of $\mathbb{R}^n$, we define the normalized local Hausdorff 
distance $d_{x,r}$ by the formula 
\[
	d_{x,r}(X,Y)=\frac{1}{r}\sup\{\dist(z,Y):z\in X\cap
	B(x,r)\}+\frac{1}{r}\sup\{\dist(y,X):y\in Y\cap B(x,r)\};
\]
a sequence $E_k\subseteq U$ converges to a set $E\subseteq U$ in local Hausdorff
distance, by definition, we mean that for any $x\in U$ and $0<r<\dist(x,U^c)$, 
$d_{x,r} (E_k,E)\to 0$ as $k\to \infty$.
\section{Convergence of quasiminimal sets}
For any $m$-plane $T$, we will denote by $T_{\natural}$ the orthogonal 
projection of $\mathbb{R}^n$ onto $T$. For any $x\in \mathbb{R}^n$ and $r>0$,
we denote by $\scale{x,r}:\mathbb{R}^n\to \mathbb{R}^n$ the mapping
given be $\scale{x,r}(y)=r^{-1}(y-x)$. For any $m$-rectifiable set $E\subseteq
\mathbb{R}^n$ and mapping $\varphi:E\to\mathbb{R}^n$, we will denote by $\ap
J_m \varphi$ the approximate Jacobian of $\varphi$, see 
\cite[Theorem 3.2.22]{Federer:1969}. 
\begin{lemma}\label{le:convergence}
	Let $\{E_k\}$ be a sequence of $m$-rectifiable subsets in $U$. Suppose that
	there is an $m$-rectifiable set $E\subseteq U$ such that 
	$\HM^m(E)= \lim_{k\to\infty}\HM^m(E_k)<+\infty$,
	and for $\HM^d$-a.e. $x\in E$, by setting $T=\Tan(E,x)$, 
	\begin{equation}\label{eq:StrCon}
		\lim_{r\to 0}\varliminf_{k\to \infty}\HM^m(T_{\natural}\circ
		\scale{x,r} ( E_k\cap B(x,r)))\geq \omega_m.
	\end{equation}
	Then we have that  
	\[
		\var(E_k)\wc \var(E).
	\]
\end{lemma}
\begin{proof}
	We first prove that for any open set $\mathcal{O}\subseteq U$, 
	\begin{equation}\label{eq:lsc}
		\HM^m(E\cap \mathcal{O})\leq \varliminf_{k\to \infty}\HM^m(E_k\cap \mathcal{O}).
	\end{equation}
	Since $E$ is rectifiable, we have that for $\HM^m$-a.e. $x\in E$, denote by
	$E^1$ the collection of such point,
	\[
		\lim_{r\to 0}\frac{\HM^m(E\cap B(x,r))}{\omega_m r^m}=1.
	\]
	For any $\varepsilon>0$ fixed, we can find $r_{x,1}>0$ such that for any
	$0<r<r_{x,1}$, 
	\[
		(1-\varepsilon)\omega_m r^m\leq \HM^m(E\cap B(x,r))\leq (1+\varepsilon)\omega_m r^m.
	\]
	We get from \eqref{eq:StrCon} that there exist $r_{x,2}>0$ and
	$k_{x}>0$ such that 
	\begin{equation}\label{eq:StrCon10}
		\HM^m(E_k\cap B(x,r))\geq r^m\HM^m(P_{T_x}\circ
		\scale{x,r} ( E_k\cap B(x,r)))\geq (1-\varepsilon)\omega_m r^{m}
	\end{equation}
	for any $0<r<r_{x,2}$ and $k\geq k_x$. We put $r_x=\min\{r_{x,1},r_{x,2}\}$,
	then 
	\[
		\HM^m(E_k\cap B(x,r))\geq (1-\varepsilon)(1+\varepsilon)^{-1}\HM^m(E\cap
		B(x,r)), \text{ for any } 0<r<r_x \text{ and }k\geq k_x.
	\]
	We see that $\mathscr{B}=\{B(x,r)\subseteq \mathcal{O}:x\in E^1\cap \mathcal{O},0<r<r_x\}$ is a 
	Vitali covering of $E^1\cap \mathcal{O}$, thus there exists a countable may balls
	$\{B_i\}_{i\in I}\subseteq \mathscr{B}$, such that $B_i\cap B_j=\emptyset$
	for $i,j\in I$, $i\neq j$, and 
	\[
		\HM^m \left(E^1\cap \mathcal{O}\setminus \bigcup_{i\in I}B_i\right)=0.
	\]
	We take $N>0$ such that $\HM^m \left(E^1\cap \mathcal{O}\setminus \cup_{i>
	N}B_i\right)<\varepsilon$. Assume that $B_i=B(x_i,r_i)$, $i\in I$. Then we
	have that, for any $k\geq \max\{k_{x_i}:1\leq i\leq N\}$, 
	\[
		\begin{aligned}
			\HM^m(E\cap \mathcal{O})&=\HM^m(E^1\cap \mathcal{O})\leq
			\sum_{i=1}^{N}\HM^m(E\cap B_i)+\varepsilon\\
			&\leq\frac{1+\varepsilon}{1-\varepsilon} \sum_{i=1}^{N}\HM^m(E_k\cap
			B_i)+\varepsilon\\
			&\leq \frac{1+\varepsilon}{1-\varepsilon} \HM^m(E_k\cap 
			\mathcal{O})+\varepsilon,
		\end{aligned}
	\]
	thus
	\[
		\HM^m(E\cap \mathcal{O}) \leq \frac{1+\varepsilon}{1-\varepsilon}
		\varliminf_{k\to \infty}\HM^m(E_k\cap \mathcal{O})+\varepsilon,
	\]
	we let $\varepsilon$ tend to $0$ to get that \eqref{eq:lsc} holds.

	Next, we show that, for any subsequence of $\{\var(E_k)\}$, if it converge
	to some varifold $V$, then 
	\[
		\Theta_{\ast}^{m}(\|V\|,x)\geq 1 \text{ for }\HM^m \text{-a.e.}x\in E,
	\]
	and so that $\Theta(\|V\|,x)=1$ for $\HM^m$-a.e. $x\in E$ and $\|V\|(U\setminus
	E)=0$. Indeed, we assume that $\var(E_{k_{\ell}})\to V$. Then for any 
	$x\in E^1$, and any ball $B(x,r)\subseteq U$, we have that 
	\begin{equation}\label{eq:StrCon30}
		\|V\|(\overline{B(x,r)})\geq \varlimsup_{\ell\to
		\infty}\HM^m\left(E_{k_{\ell}}\cap
		\overline{B(x,r)} \right)\geq \HM^m(E\cap B(x,r)),
	\end{equation}
	thus
	\[
		\Theta_{\ast}^m(\|V\|,x)\geq \Theta^m(E,x)\geq 1.
	\]
	But $\HM^m(E)=\lim_{k\to \infty}\HM^m(E_k)=\|V\|(U)$, we have so that
	$\Theta(\|V\|,x)=1$ for $\HM^m$-a.e. $x\in E$ and $\|V\|(U\setminus E)=0$.

	Finally, we show that $\VarTan(V)=\{\var(\Tan(E,x))\}$ for $\HM^m$-a.e. $x\in E$.
	We will denote by $E^2$ the points $x$ in $E_1$ that $\Theta^m(\|V\|,x)=1$
	and $E$ has unique tangent $m$-plane at $x$. Then we see that
	$\HM^m(E\setminus E^2)=0$. For any $x\in E^2$, we have that 
	\[
		1\leq \lim_{r\to 0+}\varliminf_{\ell\to \infty}\frac{\HM^m(T_{\natural}\circ
		\scale{x,r}(E_{k_{\ell}}\cap B(x,r)))}{\omega_m }\leq \lim_{r\to
		0+}\varlimsup_{\ell\to \infty}\frac{\HM^m(E_{k_{\ell}}\cap B(x,r))}{\omega_m
		r^m}\leq \Theta^m(\|V\|,x),
	\]
	but $\Theta^m(\|V\|,x)=1$, we get so that 
	\begin{equation}\label{eq:StrCon40}
		\lim_{r\to 0+}\lim_{\ell\to \infty}\HM^m(T_{\natural}\circ
		\scale{x,r}(E_{k_{\ell}}\cap B(x,r)))= \lim_{r\to
		0+}\lim_{\ell\to \infty}\HM^m(\scale{x,r}(E_{k_{\ell}}\cap B(x,r)))=\omega_m .
	\end{equation}

	We put $E_{k_{\ell},x,r}=\scale{x,r}(E_{k_{\ell}}\cap B(x,r))$, $T=\Tan(E,x)$ and
	$Q_{\ell}(y)=\Tan(E_{k_{\ell},x,r},y)$ for $y\in E_{k_{\ell},x,r}$.
	Employing \cite[8.9 (3)]{Allard:1972}, we have that 
	\[
		\|Q_{\ell}(y)-T\|=\|T^{\perp}\circ Q_{\ell}(y)\|=\sup_{v\in
		T,|v|=1}|T^{\perp}(v)|,
	\]
	so that we can find $v_1\in T$, $v=1$ such that
	$\|Q_{\ell}(y)-T\|=|T^{\perp}(v_1)|$. Let $v_1,v_2,\ldots,v_m$ be a unit
	orthogonal basis of $Q_{\ell}(y)$. Let $\Phi_{\ell}:E_{k_{\ell},x,r}\to T$ be
	defined by $\Phi_{\ell}(y)=T_{\natural}(y)$. Then 
	\[
		\ap J_m\Phi_{\ell}(y)=\|T_{\natural}(v_1)\wedge
		T_{\natural}(v_2)\wedge\cdots\wedge T_{\natural}(v_m)\|\leq |T_{\natural}(v_1)|,
	\]
	thus 
	\[
		\ap J_m\Phi_{\ell}(y)^2\leq 1- \|Q_{\ell}(y)-T\|^2.
	\]
	We get that 
	\begin{equation}\label{eq:StrCon50}
		\|Q_{\ell}(y)-T\|^2\leq 1-\ap J_m\Phi_{\ell}(y)^2\leq 2(1-\ap
		J_m\Phi_{\ell}(y)),
	\end{equation}
	and by the H\"older's inequality and Theorem 3.2.22 in \cite{Federer:1969}, we
	have that
	\[
		\begin{aligned}
			\int_{E_{k_{\ell},x,r}}\|Q_{\ell}(y)-T\|d \HM^m(y)&\leq 2\HM^m(E_{k_{\ell},x,r})
			\int_{E_{k_{\ell},x,r}}(1-\ap J_m\Phi_{\ell}(y)) \\
			&\leq 2\HM^m(E_{k_{\ell},x,r})
			\left(\HM^m(E_{k_{\ell},x,r})-\HM^m(P_T(E_{k_{\ell},x,r}))\right),
		\end{aligned}
	\]
	combine this with \eqref{eq:StrCon40}, we get that 
	\begin{equation}\label{eq:StrCon60}
		\lim_{r\to 0+}\lim_{\ell\to \infty}\int_{E_{k_{\ell},x,r}}\|Q_{\ell}(y)-T\|\ud
		\HM^m(y) =0.
	\end{equation}

	For any $C\in \VarTan(V,x)$, we assume that 
	\[
		C=\lim_{j\to \infty}(\scale{x,r_j})_{\#}V,
	\]
	where $\{r_j\}$ is a decreasing sequence which tend to 0. Then we get that
	\[
		C=\lim_{j\to \infty}\lim_{k\to \infty}\var(\scale{x,r_j}(E_k)),
	\]
	$C$ support on $T$, and 
	\begin{equation}\label{eq:StrCon70}
		C\mr B(0,1)\times G(n,m)=\lim_{j\to \infty}\lim_{k\to
		\infty}\var(F_{k,x,r_j}).
	\end{equation}
	
	For any $\varphi\in C_{c}^{\infty}(\mathbf{R}^n\times G(n,m),\mathbf{R})$
	supported on $B(0,1)\times G(n,m)$, by \eqref{eq:StrCon70}, we have that 
	\begin{equation}\label{eq:StrCon80}
		C(\varphi)=\lim_{j\to\infty}\lim_{\ell\to\infty}
		\int_{E_{k_{\ell},x,r_{j}}}\varphi(y,Q_{\ell}(y))\ud\HM^m(y),
	\end{equation}
	we let $\psi\in C_{c}^{\infty}(\mathbf{R}^n,\mathbf{R})$ be defined by
	$\psi(x)=\varphi(x,T)$, then again by \eqref{eq:StrCon70}, we have that 
	\begin{equation}\label{eq:StrCon90}
		\lim_{j\to\infty}\lim_{\ell\to\infty}\int_{E_{k_{\ell},x,r_{j}}}
		\psi(x)\ud\HM^m(y)=\int_{T}\psi(y)\ud\HM^m(y)=\var(T)(\varphi).
	\end{equation}
	We get, from \eqref{eq:StrCon60}, so that 
	\begin{equation}\label{eq:density112}
		\begin{aligned}
			\left|C(\varphi)-\var(T)(\varphi)\right|&\leq
		\lim_{j\to\infty}\lim_{\ell\to\infty}\int_{E_{k_{\ell},x,r_{j}}}
		|\varphi(y,Q_{\ell}(y))-\varphi(y,T)|\ud\HM^m(y)\\
		&\leq \|D\varphi\|_{\infty}\varlimsup_{j\to\infty}\varlimsup_{i\to\infty}
		\int_{E_{k_{\ell},x,r_{j}}}\|Q_{\ell}(y)-T\|\ud\HM^m(y)=0.
		\end{aligned}
	\end{equation}
Thus 
\[
	C\mr B(0,1)\times G(n,m) =\var(T)\mr B(0,1)\times G(n,m),
\]
but both $C$ and $\var(T)$ are cones, we have that $C=\var(T)$.
\end{proof}
Let $E\subseteq U$ be given, and let $B$ be an open ball such that
$\overline{B}\subseteq U$. A family of mappings $\{\varphi_t\}_{0\leq t 
\leq 1}$ from $E$ to $U$ is called a deformation of $E$ in $B$ if 
\begin{itemize}
	\item $\varphi_0=\id_E$, $\varphi_1$ is Lipschitz, $\varphi_t(x)=x$ for
		$x\in E\setminus B$, and
	\item $[0,1]\times E\to U$ given by $(t,x)\to \varphi_t(x)$ is continuous.
\end{itemize}
By a deformation of $E$ in $U$ we mean a deformation of $E$ in a ball which is
contained in $U$.
\begin{definition}
	For any nondecreasing function $h:[0,+\infty)\to [0,+\infty]$, and number 
	$M\geq 1$, we denote by $QM(U,M,h)$ the collection of relatively closed 
	sets $E\subseteq U$ which satisfy that
	\begin{itemize}
		\item $\HM^{m}\llcorner E$ is locally finite, 
			$\HM(E\cap B(x,r))>0$ for any $x\in E$ and some $r=r(x)>0$,
		\item for any ball $B=B(x,r)$ with $\overline{B}\subseteq U$, and any
			deformation $\{\varphi_t\}_{0\leq t\leq 1}$ of $E$ in $B$, by
			setting $W_t=\{y\in U:\varphi_t(y)\neq y\}$, we have that
			\[
				\HM^m(E\cap W_1)\leq M\HM^m(\varphi_1(E\cap W_1))+h(r)r^m.
			\]
	\end{itemize}
\end{definition}
 It is quite easy to see from the definition that, if $M_1\leq M_2$ and 
 $h_1\leq h_2$, then 
\[
	QM(U,M_1,h_1)\subseteq QM(U,M_2,h_2).
\]
If the function $h$ satisfies that $h(t)=0$ for $t<\delta$, and
$h(t)=+\infty$ for $t\geq \delta$, where $\delta>0$, then the sets in $QM(U,M,h)$
are usual $(U,M,\delta)$-quasiminimal sets, see for example Definition 2.4 in
\cite{David:2003}, and also Definition 1.9 in \cite{DS:2000}, but it is called 
$(U,M,\delta)$-quasiminimizer. If $h$ satisfies that $h(t)=h\in [0,1)$ is a 
constant for $t<\delta$, and $h(t)=+\infty$ for $t\geq \delta$, where 
$\delta>0$, then $QM(U,M,h)$ will be the general Almgren quasiminimal sets 
$GAQ(M,\delta,U,h)$ defined in Definition 2.10 in \cite{David:2009}. A function
$h:[0,\infty)\to [0,\infty]$ is called a gauge function if $h$ is a nondecreasion
function with $h(0+)=0$. Note that if $h$ is a gauge function, then $QM(U,1,h)$ 
will be the usual almost minimal sets, see for example Definition 4.3 in 
\cite{David:2009}. We see from Lemma 2.15 in \cite{David:2009} that every set 
in $QM(U,M,h)$ is local Ahlfors regular in case $h(0+)$ small enough, namely 
that \eqref{eq:AR} holds, and the constant $C$ only depends on $n$ and $m$.  
\begin{lemma}\label{le:lscHM}
	Let $\{E_k\}\subseteq QM(U,M_k,h_k)$ be a sequence. Suppose that $E_k$
	converge to some set $E$ in $U$ in local Hausdorff distance, $M=\varlimsup_{k\to\infty}M_k<+\infty$, and $h=\varlimsup_{k\to\infty}
	h_k$ satisfying that $h(0+)$ is small enough. Then we have that
\begin{enumerate}
		\item \label{lscHM1} $\HM^m(E\cap \mathcal{O})\leq \varliminf_{k\to\infty}\HM^m(E_k\cap
			\mathcal{O})$ for any open set $\mathcal{O}\subseteq U$;
		\item \label{lscHM2} $E$ is $m$-rectifiable, $E\in QM(U,M,h)$;
		\item \label{lscHM3} $\varlimsup_{k\to \infty}\HM^m(E_k\cap K)\leq (1+Ch(0+))M\HM^m(E\cap
			K)$ for any compact set $K\subseteq U$.
\end{enumerate}
\end{lemma}
\begin{proof}
Indeed, \eqref{lscHM1} follows from Lemma 3.3 in \cite{David:2009}. The fact
$E\in QM(U,M,h)$ follows from Lemma 4.1 in in \cite{David:2009}; and the 
rectifiability of $E$ comes from the local uniform rectifiability of $E$,
which can be proved by adapting the proof of the local uniform rectifiability of
quasiminimal sets (Theorem 2.11 in \cite{DS:2000}) to generalized quasiminimal
sets, see \cite[p.81]{David:2009}

It follows from Lemma 3.12 in \cite{David:2009} that 
\[
	\varlimsup_{k\to \infty}\HM^m(E_k\cap K)\leq (1+Ch(t))M\HM^m(E\cap
	K), \text{ for any compact set }K\subseteq U,
\]
for $t$ small enough which makes $h(t)$ small enough, thus we let $t$ tends to
$0$ to get the conculsion \eqref{lscHM3}.

\end{proof}
From above lemma, we see that $QM(U,M,h)$ is comapct under the locally
Hausdorff distance. That is, for any sequence $\{E_k\}\subseteq QM(U,M,h)$,
there is a subsequence $\{E_{k_{\ell}}\}$ which converges in local Hausdorff
distance to some set in $QM(U,M,h)$.
\begin{theorem} \label{thm:QM}
	Let $\{E_k\}$ be a sequence of sets such that $E_k\in QM(U,M_k,h_k)$.
	Suppose that $E_k\to E$ in $U$, $M=\varlimsup_{k\to\infty}M_m <+\infty$,
	$h=\varlimsup_{m\to\infty} h_m$ satisfy that $h(0+)$ is small enough. 
	If $\HM^m(E)=\lim_{k\to\infty}\HM^m(E_k)<\infty$, then $\var(E_k)\wc \var(E)$.
\end{theorem}
\begin{proof}
	By Lemma \ref{le:lscHM}, we have that $E$ is rectifiable. Thus for
	$\HM^m$-a.e. $x\in E$, $\Theta^m(E,x)=1$ and $E$ has a tangent plane at $x$, 
	denote it by $T_x$.
	Since $E_k\to E$ in $U$, and $T_x=\Tan(E,x)$, we get that for any
	$\varepsilon>0$, there exist $0<r_{\varepsilon}<\dist(x,U^c)$ and
	$k_{\varepsilon}>0$ such that for any $0<r<r_{\varepsilon}$ and 
	$k\geq k_{\varepsilon}$,we have that 
	\begin{equation}\label{eq:QM20}
		\scale{x,r}(E_k\cap B(x,r)) \subseteq T_x+B(0,\varepsilon) 
	\end{equation}
	and 
	\begin{equation}\label{eq:QM21}
		(1-\varepsilon)\omega_m r^m\leq \HM^m(E\cap B(x,r))\leq (1+\varepsilon)
		\omega_m r^m.
	\end{equation}
	Since $\HM^m(E\cap \partial B(x,r))=0$ for $\HM^1$-a.e. $r>0$, we always put
	ourself in the case for such $r$.
	We put $T=T_x$, $E_{k,x,r}=\scale{x,r}(E_k\cap B(x,r))$ and define $h_{k,r}$ 
	by given $h_{k,r}(t)=h_k(rt)$. Then we have that 
	\begin{equation}\label{eq:QM30}
		E_{k,x,r}\in QM\left(B(0,1),M_k,h_{k,r}\right).
	\end{equation}

	For any $0<\varepsilon< 1$, we let $g:\mathbb{R}\to
	\mathbb{R}$ be a function of class $C^{\infty}$ such that $0\leq g\leq 1$,
	$g(t)=1$ for $t\leq \varepsilon$, $g(t)=0$ for $t\geq 1$, and $\|Dg\|\leq
	2/\varepsilon$. We define
	mapping $T_{\natural}^{\varepsilon}:\mathbb{R}^n\to \mathbb{R}^n$ by
	\[
		T_{\natural}^{\varepsilon}(x)=(1-g(|x|))x+g(|x|)T_{\natural}(x).
	\]
	Then, by setting $T^{\varepsilon}=[T+B(0,\varepsilon)]\cap B(0,1)$, we have
	that
	\begin{equation}\label{eq:QM40}
		\Lip \left(T_{\natural}^{\varepsilon}\vert_{T^{\varepsilon}}\right)\leq
		\|DT_{\natural}^{\varepsilon}\vert_{T^{\varepsilon}}\|
		\leq 2+\frac{2 \varepsilon}{\varepsilon} =4.
	\end{equation}

	We claim that 
	\begin{equation}\label{eq:QM50}
		T_{\natural}^{\varepsilon}(E_{k,x,r})\supseteq T\cap B(0,1-2 \varepsilon).
	\end{equation}
We proceed by contradiction for the claim. Assume $y\in T\cap B(0,1-2
\varepsilon)\setminus T_{\natural}^{\varepsilon}(E_{k,x,r})$. Then there is a
small ball $B(y,\rho)$ such that $T_{\natural}^{\varepsilon}(E_{k,x,r})\cap
B(y,\rho)=\emptyset$. Let $\Psi:\mathbb{R}^n\to \mathbb{R}^n $ be a mapping of
class $C^{\infty}$ such that $\Psi(z)\in \partial B(0,1-2 \varepsilon)$ for
$z\in B(0,1-2 \varepsilon)\setminus B(y,\rho)$, and $\Psi(z)=z$ for $z\not\in
B(0,1-2 \varepsilon)$. Then, by setting $A_{x,r,\varepsilon}=E_k\cap
\overline{B(x,r)}\setminus B(x,(1-2 \varepsilon)r)$, we have that 
\[
	\begin{aligned}
		\HM^m \left(\Psi\circ T_{\natural}^{\varepsilon}(E_{k,x,r})\right)&=
		\HM^m \left(T_{\natural}^{\varepsilon}(E_{k,x,r}\setminus B(0,1-2
		\varepsilon))\right)\leq
		\Lip(T_{\natural}^{\varepsilon}\vert_{T^{\varepsilon}} )^m 
		\HM^m \left(E_{k,x,r}\setminus B(0,1-2 \varepsilon)\right)\\
		&\leq 4^m r^{-m} \HM^m(E_k\cap
		A_{x,r,\varepsilon}),
	\end{aligned}
\]
thus
\[
	\begin{aligned}
		\varlimsup_{k\to \infty}\HM^m \left(\Psi\circ
		T_{\natural}^{\varepsilon}(E_{k,x,r})\right)&\leq 4^m r^{-m} \varlimsup_{k\to \infty}\HM^m(E_k\cap
		A_{x,r,\varepsilon})\\
		&\leq4^m
		r^{-m}(1+Ch(0+))M\HM^m(E\cap A_{x,r,\varepsilon}).
	\end{aligned}
\]
Hence 
\[
	\begin{aligned}
		\HM^m(E\cap B(x,r))&\leq \varliminf_{k\to \infty}\HM^m(E_k\cap B(x,r))=
		r^{m}\varliminf_{k\to \infty}\HM^m(E_{k,x,r})\\
		&\leq r^{m}\varliminf_{k\to \infty } \left(M\HM^m \left(\Psi\circ
		T_{\natural}^{\varepsilon}(E_{k,x,r})\right)+h_{k,r}(1)\right)\\
		&\leq  4^m
		(1+Ch(0+))M^2\HM^m(E\cap A_{x,r,\varepsilon})+ h(0+)r^m.
	\end{aligned}
\]
But $\HM^m(E\cap B(x,r))\geq (1-\varepsilon)\omega_mr^m$ and 
\[
	\HM^m(E\cap A_{x,r,\varepsilon})\leq (1+\varepsilon)\omega_m r^m-
	(1-\varepsilon)\omega_m (1-2 \varepsilon)^mr^m\leq 2(m+1) \varepsilon r^m,
\]
we get so that 
\[
	(1- \varepsilon)\omega_m\leq 2(m+1)\cdot 4^m
		(1+Ch(0+))M^2 \varepsilon+ h(0+),
\]
but this is a contradiction when $h(0+)$ is small enough and $\varepsilon$ tends
to $0$, and we proved the claim.

By \eqref{eq:QM50}, we have that 
\[
	\begin{aligned}
		\HM^m(T_{\natural}\circ \scale{x,r}(E_k\cap B(x,r)))&\geq
		\HM^m(T_{\natural}(E_{k,x,r}\cap B(0,1-\varepsilon)))
		=\HM^m(T_{\natural}^{\varepsilon}(E_{k,x,r}\cap B(0,1-\varepsilon)))\\
		&\geq \HM^m(T_{\natural}^{\varepsilon}(E_{k,x,r}))-
		\HM^m(T_{\natural}^{\varepsilon}(E_{k,x,r}\setminus B(0,1-\varepsilon)))\\
		&\geq \HM^m(T\cap B(0,1-2 \varepsilon))-4^m 
		r^{-m} \HM^m(E_k\cap A_{x,r,\varepsilon}),
	\end{aligned}
\]
thus 
\[
	\begin{aligned}
		\varliminf_{k\to \infty}\HM^m(T_{\natural}\circ \scale{x,r}(E_k\cap
		B(x,r))) & \geq (1-2 \varepsilon)\omega_m-4^m 
		r^{-m}(1+Ch(0+))M\HM^m(E\cap A_{x,r,\varepsilon})\\
		&\geq (1-2 \varepsilon)\omega_m-4^m (1+Ch(0+))M
		\cdot 2(m+1) \varepsilon,
	\end{aligned}
\]
and 
\[
	\lim_{r\to 0}\varliminf_{k\to \infty}\HM^m(T_{\natural}\circ
	\scale{x,r}(E_k\cap B(x,r)))\geq (1-2 \varepsilon)\omega_m-
	4^m (1+Ch(0+))M \cdot 2(m+1) \varepsilon,
\]
we let $\varepsilon$ tend to $0$ to get that 
\[
	\lim_{r\to 0}\varliminf_{k\to \infty}\HM^m(T_{\natural}\circ
	\scale{x,r}(E_k\cap B(x,r)))\geq \omega_m.
\]
Applying Lemma \ref{le:convergence}, we get the conclusion $\var(E_k)\wc
\var(E)$.
\end{proof}
\begin{corollary}
	Let $\{E_k\}$ be a sequence of sets such that $E_k\in QM(U,M_k,h_k)$.
	Suppose that $E_k\to E$, $M=\varlimsup_{k\to\infty}M_m =1$,
	$h=\varlimsup_{m\to\infty} h_m$ satisfy that $h(0+)=0$.
	Then we have that $\var(E_k)\wc \var(E)$. In particular, for any gauge
	function $h$, the mapping $QM(U,1,h)\to \Var_m(U)$ given by $E\mapsto
	\var(E)$ is a homeomorphism between its domain and image, where $QM(U,1,h)$
	is equipped with the topology deduced by the local Hausdorff distance and
	$\Var_m(U)$ is equipped with the weak topology.
\end{corollary}
\begin{proof}
	By Lemma \ref{le:lscHM}, we have that for any open set $\mathcal{O}$ and
	compact set $K$, 
	\[
		\HM^m(E\cap \mathcal{O})\leq \varliminf_{k\to\infty}\HM^m(E_k\cap \mathcal{O})
	\]
	and 
	\[
		\HM^m(E\cap K)\geq \varlimsup_{k\to\infty}\HM^m(E_k\cap K).
	\]
	If $\mathcal{O}\subseteq U$ is an open set satisfying that
	$\overline{\mathcal{O}}\subseteq U$ and $\HM^m(E\cap \partial
	\mathcal{O})=0$, then we have that 
	\[
		\HM^m(E\cap \overline{\mathcal{O}})\geq
		\varlimsup_{k\to\infty}\HM^m(E_k\cap \overline{\mathcal{O}})\geq
		\varliminf_{k\to\infty}\HM^m(E_k\cap \mathcal{O})\geq \HM^m(E\cap
		\mathcal{O}),
	\]
	thus 
	\[
		\HM^m(E\cap \mathcal{O})= \lim_{k\to\infty}\HM^m(E_k\cap \mathcal{O}).
	\]

	For any $x\in U$, we see that $\HM^m(E\cap \partial B(x,r))=0$ for
	$\HM^1$-a.e. $r>0$, we can find $r>0$ so that $B(x,r)\subseteq U$,
	$\HM^m(E\cap \partial B(x,r))=0$ and $\HM^m(E\cap B(x,r))<+\infty$, thus 
	\[
		\HM^m(E\cap B(x,r))= \lim_{k\to\infty}\HM^m(E_k\cap B(x,r)).
	\]
	By Theorem \ref{thm:QM}, we have that $\var(E_k\cap B(x,r))\wc \var(E\cap
	B(x,r))$. Hence 
	\[
		\var(E_k)\wc \var(E).
	\]

\end{proof}

\section{Convergence of quasiminimal sets involving elliptic integrands}
A function $F:\mathbb{R}^n\times \grass{n}{m}\to (0,\infty)$ is called an
integrand, if additionally  $1\leq \sup F/\inf F<+\infty$, then we say that
$F$ is bounded. For any $x\in \mathbb{R}^n$, we define integrand $F^x$ be
given $F^x(y,T)=F(x,T)$. We define the functional
$\Phi_F:\Var(\mathbb{R}^n)\to \mathbb{R}$ by the formula 
\[
	\Phi_F(V)=\int F(x,T)\ud V(x,T).
\]
An integrand $F$ is called elliptic if there exists a continuous function
$c:\mathbb{R}^n\to (0,\infty)$ such that for any $x\in \mathbb{R}^n$, 
\begin{equation}\label{eq:elliptic}
	\Phi_{F^x}(S)-\Phi_{F^x}(D)\geq c(x)(\HM^m(S)-\HM^m(D))
\end{equation}
whenever $D=T\cap B(0,1)$ for some $T\in \grass{n}{m}$ and $S$ is a compact 
$m$-rectifiable set which can not be mapped
into $T\cap \partial B(0,1)$ by any Lipschitz mapping which leaves $T\cap
\partial B(0,1)$ fixed, see \cite{Almgren:1974}. 
$F$ is called semi-elliptic if it hold  $\Phi_{F^x}(S)-\Phi_{F^x}(D)\geq 0$ 
instead of \eqref{eq:elliptic}.

\begin{lemma}\label{le:lscHMF}
	Let $E_k$, $M_k$, $h_k$, $E$, $M$ and $h$ be the same as in Lemma
	\ref{le:lscHM}, and let $F$ be a semi-elliptic integrand. Then, for any open
	set $\mathcal{O}\subseteq U$, we have that 
	\[
		\Phi_F(E\cap \mathcal{O})\leq \varliminf_{k\to\infty}\Phi_F(E_k\cap \mathcal{O})
	\]
\end{lemma}
\begin{proof}
	For a proof, see for example Theorem 25.7 in \cite{David:2014} or
	Theorem 2.5 in \cite{Fang:2013}, so we omit it here.
\end{proof}
\begin{theorem} \label{thm:QMF}
	Let $\{E_k\}$ be a sequence of sets such that $E_k\in QM(U,M_k,h_k)$.
	Suppose that $E_k\to E$ in $U$, $M=\varlimsup_{k\to\infty}M_m <+\infty$, and
	$h=\varlimsup_{m\to\infty} h_m$ satisfy that $h(0+)$ is small enough. 
	If $\Phi_F(E)=\lim_{k\to\infty}\Phi_F(E_k)<\infty$ for some elliptic
	integrand $F$, then $\var(E_k)\wc \var(E)$.
\end{theorem}
\begin{proof}
	For any $x\in \mathbb{R}^n$ and $r>0$, we define integrand
	$F_{x,r}$ by given 
	\[
		F_{x,r}(y,T)=F(\scale{x,r}(y),T), \text{ for } (y,T)\in 
		\mathbb{R}^n\times \grass{n}{m}.
	\]
	Then $F_{x,r}$ is also elliptic, and $F^x=\lim_{r\to 0}F_{x,r}$. Since
	$\Phi_F(E)=\lim_{k\to\infty}\Phi_F(E_k)<\infty$, we see that 
	\[
		\Phi_{F_{x,r}}(\scale{x,r}(E))=\lim_{k\to\infty}\Phi_{F_{x,r}}
		(\scale{x,r}(E_k))<\infty.
	\]
	We will put $U_{x,r}=\scale{x,r}(U)$, $h_{k,r}(t)=h(rt)$, $B_{x,r}=B(x,r)$,
	$E_{k,x,r}=\scale{x,r}(E_k\cap B(x,r))$, $E_{x,r}=\scale{x,r}(E)$,
	and $B=B(0,1)$ for convenient. Then $\scale{x,r}(E_k)\in
	QM(U_{x,r},M,h_{k,r})$ and $E_{k,x,r}\in QM(B,M,h_{k,r})$. By Lemma \ref{le:lscHMF}, we have that  
	\[
		\Phi_{F_{x,r}}\left(\scale{x,r}(E)\cap \left(U_{x,r}\setminus
		\overline{B}\right)\right)\leq \varliminf_{k\to \infty}
		\Phi_{F_{x,r}}\left(\scale{x,r}(E_k)\cap \left(U_{x,r}\setminus
		\overline{B}\right)\right),
	\]
	thus 
	\begin{equation}\label{eq:QMF10}
		\Phi_{F_{x,r}}\left(\scale{x,r}(E)\cap \overline{B}\right)\geq
		\varlimsup_{k\to \infty}\Phi_{F_{x,r}}\left(\scale{x,r}(E_k)\cap
		\overline{B}\right).
	\end{equation}

	We see that for $\HM^m$-a.e $x\in E$, $\Tan(E,x)$ exists and
	$\Theta^m(E,x)=1$. For any $\varepsilon>0$, we take
	$0<r_{\varepsilon}<\dist(x,U^c) $ and $k_{\varepsilon}>0$ such that, for any
	$0<r<r_{\varepsilon}$ and $k\geq k_{\varepsilon}$, 
	\begin{equation}\label{eq:QMF20}
		\scale{x,r}(E_k\cap B(x,r)) \subseteq \Tan(E,x)+B(0,\varepsilon) 
	\end{equation}
	and 
	\begin{equation}\label{eq:QMF21}
		(1-\varepsilon)\omega_m r^m\leq \HM^m(E\cap B(x,r))\leq (1+\varepsilon)
		\omega_m r^m.
	\end{equation}

	Let $g_1:\mathbb{R}\to \mathbb{R}$ be a function of class $C^{\infty}$ such
	that $0\leq g_1\leq 1$, $g_1(t)=0$ for $t\in (-\infty, 1-3 \varepsilon]\cup
	[1,+\infty)$, $g_1(t)=1$ for $t\in [1-2 \varepsilon,1-\varepsilon]$, and
	$\|Dg_1\|\leq 2/\varepsilon$. We let $\Pi^{\varepsilon}:\mathbb{R}^n\to
	\mathbb{R}^n$ be the mapping defined by
	\[
		\Pi^{\varepsilon}(x)=(1-g_{1}(|x|))x+g_1(|x|)T_{\natural}(x),
	\]
	take $1-2 \varepsilon< \rho <\sqrt{1-2 \varepsilon}$ and
	$\overline{E}_k=\Pi^{\varepsilon}(E_{k,x,r})\cap \overline{B(0,\rho)}$.
	We claim that $\overline{E}_k\supseteq \partial B(0,\rho)\cap T$ and
	$\overline{E}_k$ cannot be mapped into $\partial B(0,\rho)\cap T$ by any
	Lipschitz mapping which leaves $\partial B(0,\rho)\cap T$ fixed, where
	$T=\Tan(E,x)$ and $k\geq k_{\varepsilon}$. Suppose for the sake of 
	contradiction there is Lipschitz mapping $\varphi$ such that
	$\varphi\vert_{B(0,\rho)^c}=\id$ and $\varphi(\overline{E}_k)\subseteq T\cap
	\partial B(0,\rho)$. Indeed, by putting $T^{\varepsilon}=[T+B(0, 
	\varepsilon)] \cap B(0,1)$ and $\overline{\varphi}=\varphi\circ
	\Pi^{\varepsilon}\circ \scale{x,r}$, we have that  
	\[
		\Lip(\Pi^{\varepsilon}\vert_{T^{\varepsilon}})\leq 4,
	\]
	and 
	\[
		\begin{aligned}
			\HM^m(\overline{\varphi}(E_k\cap
			B(x,r)))&=\HM^m(\Pi^{\varepsilon}\circ
			\scale{x,r}(E_k\cap B(x,r))\setminus B(0,\rho))\\
			&\leq 4^mr^{-m}\HM^m(E_k\cap B(x,r)\setminus 
			B(x,(1-3 \varepsilon)r)),
		\end{aligned}
	\]
	thus by Lemma \ref{le:lscHM}, we have that 
	\[
		\begin{aligned}
			\HM^m(E\cap B(x,r))&\leq \varliminf_{k\to \infty}\HM^m(E_k\cap
			B(x,r))\leq 
			\varliminf_{k\to \infty}(Mr^m\HM^m(\overline{\varphi}(E_k\cap
			B(x,r)))+h_k(r)r^m)\\
			&\leq (4^mM^2(1+Ch(0+))(3m+3)\omega_m \varepsilon +h(r))r^m,
		\end{aligned}
	\]
	since
	\[
		\varlimsup_{k\to \infty}\HM^m(E_k\cap A_{\varepsilon})\leq 
		M(1+Ch(0+))\HM^m(E\cap A_{\varepsilon})\leq M(1+Ch(0+)) (3m+2)\omega_m
		\varepsilon;
	\]
	this contradict with the local Ahlfors regularity of $E$ in case $h(0+)$
	small enough, and the claim is true.
	
	We continue to do the estimation, in fact we would like to get the same
	estimation as in \eqref{eq:StrCon40}, then we use the same technique to get
	the varifold convergence. For convenient, we denote by $X\,\triangle\, Y$ 
	the symmetric difference $(X\setminus Y) \cup (Y\setminus X)$ for any sets
	$X,Y\subseteq \mathbb{R}^n$. Then we have that 
	\[
		\begin{aligned}
			\HM^m \left(\Pi^{\varepsilon}(\scale{x,r}(E_k))\,\triangle\,
			\scale{x,r}(E_k) \right)&\leq
			\left(\Lip(\Pi^{\varepsilon}\vert_{T^{\varepsilon}})^{m}+1\right)\HM^m
			\left(\scale{x,r}(E_k)\cap B\setminus B(0,1-3 \varepsilon)\right)\\
			&\leq (4^m+1)r^{-m}(E_{k}\cap B(x,r)\setminus B(x,(1-3 \varepsilon)r)),
		\end{aligned}
	\]
	thus 
	\[
		\varlimsup_{k\to \infty}\HM^m \left(\Pi^{\varepsilon}(\scale{x,r}(E_k))\,\triangle\,
		\scale{x,r}(E_k) \right)\leq (4^m+1)\cdot M(1+Ch(0+))(3m+2) \omega_m \varepsilon
	\]
	and
	\begin{equation}\label{eq:QM40}
		\begin{aligned}
			\varlimsup_{k\to \infty} \Phi_{F_{x,r}}(\overline{E}_k)&\leq
			\varlimsup_{k\to \infty}\Phi_{F_{x,r}}(\Pi^{\varepsilon}\circ 
			\scale{x,r}(E_k)\cap B)\\
			&\leq \varlimsup_{k\to \infty}\Phi_{F_{x,r}}
			(E_{k,x,r})+(\sup F)(4^m+1) M(1+Ch(0+))(3m+2) \omega_m \varepsilon\\
			&\leq \Phi_{F_{x,r}}( \scale{x,r}(E)\cap \overline{B})+C_1 \varepsilon,
		\end{aligned}
	\end{equation}
	where $C_1=(\sup F)(4^m+1) M(1+Ch(0+)) (3m+2) \omega_m$.
	On the other hand, we see from Theorem (1)(a) in Section 3.5 in
	\cite{Allard:1972} that for $\HM^m$-a.e. $x\in E$, 
	\[
		\lim_{r\to 0} r^{-m}\int_{E\cap B(x,r)} F^x(\Tan(E,y)) \ud \HM^m(y) =
		F^x(\Tan(E,x)) \omega_m,
	\]
	we get that 
	\[
		\lim_{r\to 0} \Phi_{F^x}(\scale{x,r}(E)\cap \overline{B})=
		\lim_{r\to 0} \Phi_{F^x}(\scale{x,r}(E)\cap B)=F^x(\Tan(E,x)) \omega_m.
	\]
	We put $\omega(x,r)=\sup\{|F(y,S)-F(x,S)|:|y-x|\leq r, S\in \grass{n}{m}\}$.
	Then $\omega(x,r)\to 0$ as $r\to 0$; and 
	\[
		|\Phi_{F^{x}}(\scale{x,r}(E_k)\cap B)-\Phi_{F_{x,r}}(\scale{x,r}(E_k)\cap
		B)|\leq \omega(x,r)r^{-m}\HM^m(E_k\cap B(x,r)).
	\]
	We get so that 
	\[
		\begin{aligned}
			\varlimsup_{k\to \infty}\Phi_{F^{x}}(\overline{E}_k)&\leq  \omega(x,r)
			\varlimsup_{k\to \infty}\HM^m(\overline{E}_k)+
			\varlimsup_{k\to \infty} \Phi_{F_{x,r}}(\overline{E}_k)\\
			&\leq \Phi_{F_{x,r}}( E_{x,r}\cap \overline{B})+C_1 \varepsilon+ 
			\omega(x,r) \varlimsup_{k\to \infty}\HM^m(\overline{E}_k)\\
			&\leq \Phi_{F^x}( E_{x,r}\cap \overline{B})+C_1 \varepsilon+ 
			\omega(x,r) \varlimsup_{k\to \infty}\HM^m(\overline{E}_k).
		\end{aligned}
	\]
	Since $F$ is elliptic, we have that 
	\[
		\Phi_{F^{x}}(\overline{E}_k)-\Phi_{F^{x}}(T\cap B(0,\rho))\geq
		c(x)(\HM^m(\overline{E}_k)-\HM^m(T\cap B(0,\rho))),
	\]
	thus
	\[
		\HM^m(\overline{E}_k)\leq \omega_m\rho^m+c(x)^{-1}(\Phi_{F^{x}}
		(\overline{E}_k)-F(x,T)\omega_m\rho^m).
	\]
	Hence 
	\[
		\begin{aligned}
			\varlimsup_{k\to \infty}\HM^m(\overline{E}_k)&\leq\omega_m\rho^m-c(x)^{-1}
			F(x,T)\omega_m\rho^m\\
			&\quad +c(x)^{-1}[\Phi_{F^{x}} (E_{x,r}\cap
			\overline{B})+C_1 \varepsilon+ \omega(x,r) \varlimsup_{k\to
			\infty}\HM^m(\overline{E}_k)],
		\end{aligned}
	\]
	and 
	\[
		\begin{aligned}
		(1-c(x)^{-1}\omega(x,r))\varlimsup_{k\to \infty}\HM^m(\overline{E}_k)&\leq
		\omega_m\rho^m +c(x)^{-1}[\Phi_{F^{x}} (E_{x,r}\cap
		\overline{B})-	F(x,T)\omega_m\rho^m+C_1 \varepsilon]\\
			&\leq \omega_m\rho^m +c(x)^{-1}F(x,T)\omega_m(1-\rho^m)+
			c(x)^{-1}C_1 \varepsilon\\
			&\leq \omega_m +c(x)^{-1}F(x,T)\omega_m \cdot 2m \varepsilon+c(x)^{-1}C_1
			\varepsilon.
		\end{aligned}
	\]
	Thus 
	\[
		\begin{aligned}
			\HM^m(\scale{x,r}(E_k)\cap B)&\leq \HM^m(\scale{x,r}(E_k)\cap
			B\setminus B(1-3 \varepsilon))+\HM^m(\scale{x,r}(E_k)\cap B(0,1-3
			\varepsilon))\\
			&\leq \HM^m(\scale{x,r}(E_k)\cap B\setminus B(1-3
			\varepsilon))+\HM^m(\overline{E}_k),
		\end{aligned}
	\]
	and 
	\[
		\begin{aligned}
			\varlimsup_{k\to \infty}\HM^m(\scale{x,r}(E_k)\cap B)&\leq\varlimsup_{k\to
			\infty}\HM^m(\overline{E}_k) + M(1+Ch(0+))(3m+2)\omega_m \varepsilon\\
			&\leq [1-c(x)^{-1}\omega(x,r)]^{-1}(\omega_m +c_1(x)
			\varepsilon)+C_2 \varepsilon,
		\end{aligned}
	\]
	where $c_1(x)=c(x)^{-1}F(x,T)\omega_m \cdot 2m+c(x)^{-1}C_1$, and $C_2=
	M(1+Ch(0+))(3m+2)\omega_m$.
	We get so that 
	\[
		\lim_{r\to 0}\varlimsup_{k\to \infty}\HM^m(\scale{x,r}(E_k)\cap B)\leq
		\omega_m +c_1(x) \varepsilon)+C_2 \varepsilon,
	\]
	let $\varepsilon$ tend to 0, we  get that 
	\[
		\lim_{r\to 0}\varlimsup_{k\to \infty}\HM^m(\scale{x,r}(E_k)\cap B)\leq
		\omega_m.
	\]
	But we see from the proof of Theorem \ref{thm:QM} that 
	\[
		\lim_{r\to 0}\varliminf_{k\to \infty}\HM^m(T_{\natural}\circ\scale{x,r}(E_k\cap
		B(x,r)))\geq \omega_m,
	\]
	we get so that for $\HM^m$-a.e. $x\in E$.
	\[
		\lim_{r\to 0}\lim_{k\to \infty}\HM^m(\scale{x,r}(E_k\cap
		B(x,r)))=\lim_{r\to 0}\lim_{k\to \infty}\HM^m(T_{\natural}
		\circ\scale{x,r}(E_k\cap B(x,r)))=\omega_m.
	\]
	Then similar to the proof of Lemma \ref{le:convergence}, we conclude that 
	\[
		\var(E_k)\wc \var(E).
	\]
\end{proof}

\end{document}